\documentclass{article}
\usepackage[utf8]{inputenc}
\usepackage{amsmath,amsthm,amssymb,amsfonts}
\usepackage{verbatim}
\usepackage{graphicx}
\usepackage{stmaryrd} 
\usepackage{hyperref}

\usepackage[colorinlistoftodos]{todonotes}

\theoremstyle{plain}
\newtheorem{thm}{Theorem}
\newtheorem{lem}[thm]{Lemma}
\newtheorem{prop}[thm]{Proposition}

\newtheorem{remark}[thm]{Remark}
\newtheorem{remk}[thm]{Remark}

\theoremstyle{definition}
\newtheorem{definition}[thm]{Definition}
\newtheorem{defn}[thm]{Definition}
\newtheorem{exl}[thm]{Example}

\numberwithin{thm}{section}

\newcommand{\adj}{\leftrightarrow}
\newcommand{\adjeq}{\leftrightarroweq}

\DeclareMathOperator{\id}{id}
\DeclareMathOperator{\Fix}{Fix}
\def\Z{{\mathbb Z}}
\def\N{{\mathbb N}}
\def\R{{\mathbb R}}

\title{Cold and Freezing Sets in the Digital Plane}
\author{Laurence Boxer
\thanks{Department of Computer and Information Sciences, Niagara University,
        New York, 14109, USA;
        and Department of Computer Science and Engineering, State University
        of New York at Buffalo;
        boxer@niagara.edu}
}
\date{}

\begin{document}

\maketitle
\begin{abstract}
    Cold sets and freezing sets belong to the theory of (approximate) fixed points for
    continuous self-maps on digital images. We study some properties of cold sets for
    digital images in the digital plane, and we examine some relationships between cold sets
    and freezing sets.
    

\end{abstract}

\section{Introduction}
Digital topology is concerned with exploring topological
and geometric properties of digital images as stored in
computer memory, i.e., as sets of discrete pixels, usually
treated as graphs in which some notion of ``nearness"
determines adjacency. Pioneering publications in the discipline 
include~\cite{RosenfeldMAA,Rosenfeld,KongEtal, KongKop}.
Considerable success has been obtained in showing that 
digital images and the Euclidean objects they represent are
often similar with respect to properties such as connectedness, 
fundamental group, contractibility, retraction, et al.
However, the discrete and usually finite nature of a graph 
often constricts continuous functions on digital images in ways 
unmatched by similar limitations for continuous
functions on Euclidean objects. Among these restrictions are 
those associated with cold sets and freezing sets.

Cold sets and freezing sets were introduced in~\cite{BxFpSets} in order to study properties
of fixed points and approximate fixed points in digital topology. Subsequent
papers~\cite{BxConvex,BxSubsets} developed our understanding of freezing sets. In this paper,
we give more attention to cold sets.

Among our results are the following.
\begin{itemize}
    \item Multiple results concerning points that must belong to a cold set for a given digital image:
          section~\ref{essentialPtSec} and Theorem~\ref{c1ConvexVertex}.
    \item Results for which cold sets and freezing sets coincide: Theorem~\ref{c1coldEquivFreeze}
          and Proposition~\ref{c2coldEquivFreeze}. These augment a result 
          of~\cite{BxFpSets} saying that for a digital image that is {\em rigid}, i.e., the only
          continuous self-map homotopic to the identity is the identity~\cite{hmps}, cold and
          freezing are equivalent.
          In general, {\em cold} and {\em freezing} are not equivalent~\cite{BxFpSets}.
\end{itemize}
Some of our results concerning when cold and freezing are equivalent 
          show that like freezing (0-cold) sets, cold (1-cold) sets are often found
          in the boundary of $X$. For $s > 1$, $s$-cold sets may be found in the
          boundary of $X$, as in Example~\ref{c2Square-n-cold}, or in
          the interior of $X$~\cite{BxFpSets}.

\section{Preliminaries}
Let $\N$ denote the set of natural numbers; $\N^* = \{0\} \cup \N$, the
set of nonnegative integers; $\Z$, the set of integers; and $\R$, the set of real numbers.
$\#X$ will be used for the number of members of a set~$X$.

\subsection{Adjacencies}
Material in this section is largely quoted or paraphrased from~\cite{bs19a}.

A digital image is a pair $(X,\kappa)$ where
$X \subset \Z^n$ for some $n$ and $\kappa$ is
an adjacency on $X$. Thus, $(X,\kappa)$ is a graph
for which $X$ is the vertex set and $\kappa$ 
determines the edge set. Usually, $X$ is finite,
although there are papers that consider infinite $X$. Usually, adjacency reflects some type of
``closeness" in $\Z^n$ of the adjacent points.
When these ``usual" conditions are satisfied, one
may consider a subset $Y$ of $\Z^n$ containing $X$ as a model of a
black-and-white ``real world" image in which
the black points (foreground) are represented by 
the members of $X$ and the white points 
(background) by members of $Y \setminus \{X\}$.

We write $x \adj_{\kappa} y$, or $x \adj y$ when
$\kappa$ is understood or when it is unnecessary to
mention $\kappa$, to indicate that $x$ and
$y$ are $\kappa$-adjacent. Notations 
$x \adjeq_{\kappa} y$, or $x \adjeq y$ when
$\kappa$ is understood, indicate that $x$ and
$y$ are $\kappa$-adjacent or are equal.

The most commonly used adjacencies are the
$c_u$ adjacencies, defined as follows.
Let $X \subset \Z^n$ and let $u \in \Z$,
$1 \le u \le n$. Then for points
\[x=(x_1, \ldots, x_n) \neq (y_1,\ldots,y_n)=y\]
we have $x \adj_{c_u} y$ if and only if
\begin{itemize}
    \item for at most $u$ indices $i$ we have
          $|x_i - y_i| = 1$, and
    \item for all indices $j$, $|x_j - y_j| \neq 1$
          implies $x_j=y_j$.
\end{itemize}

The $c_u$-adjacencies are often denoted by the
number of adjacent points a point can have in the
adjacency. E.g.,
\begin{itemize}
\item in $\Z$, $c_1$-adjacency is 2-adjacency;
\item in $\Z^2$, $c_1$-adjacency is 4-adjacency and
      $c_2$-adjacency is 8-adjacency;
\item in $\Z^3$, $c_1$-adjacency is 8-adjacency,
      $c_2$-adjacency is 18-adjacency, and 
      $c_3$-adjacency is 26-adjacency.
\end{itemize}
In this paper, we mostly use the $c_1$ and $c_2$ adjacencies in $\Z^2$.

Let $x \in (X,\kappa)$. We use the notations
\[  N(X,x,\kappa) = \{ y \in X \, | \, y \adj_{\kappa} x \}
\]
and
\[  N^*(X,x,\kappa) =  \{ y \in X \, | \, y \adjeq_{\kappa} x \} = N(X,x,\kappa) \cup \{x\}.
\]

We say $\{x_n\}_{n=0}^k \subset (X,\kappa)$ is a {\em $\kappa$-path} (or a {\em path} if $\kappa$ is understood)
from $x_0$ to $x_k$ if $x_i \adjeq_{\kappa} x_{i+1}$ for $i \in \{0,\ldots,k-1\}$, and $k$ is the {\em length} of the path.

A subset $Y$ of a digital image $(X,\kappa)$ is
{\em $\kappa$-connected}~\cite{Rosenfeld},
or {\em connected} when $\kappa$
is understood, if for every pair of points $a,b \in Y$ there
exists a $\kappa$-path in $Y$ from $a$ to $b$.

\subsection{Digitally continuous functions}
Material in this section is largely quoted or paraphrased from~\cite{bs19a}.

We denote by $\id$ or $\id_X$ the
identity map $\id(x)=x$ for all $x \in X$.

\begin{definition}
{\rm \cite{Rosenfeld, Bx99}}
Let $(X,\kappa)$ and $(Y,\lambda)$ be digital
images. A function $f: X \to Y$ is 
{\em $(\kappa,\lambda)$-continuous}, or
{\em digitally continuous} when $\kappa$ and
$\lambda$ are understood, if for every
$\kappa$-connected subset $X'$ of $X$,
$f(X')$ is a $\lambda$-connected subset of $Y$.
If $(X,\kappa)=(Y,\lambda)$, we say a function
is {\em $\kappa$-continuous} to abbreviate
``$(\kappa,\kappa)$-continuous."
\end{definition}

\begin{thm}
{\rm \cite{Bx99}}
A function $f: X \to Y$ between digital images
$(X,\kappa)$ and $(Y,\lambda)$ is
$(\kappa,\lambda)$-continuous if and only if for
every $x,y \in X$, if $x \adj_{\kappa} y$ then
$f(x) \adjeq_{\lambda} f(y)$.
\end{thm}

A function $f: (X,\kappa) \to (Y,\lambda)$ is
an {\em isomorphism} (called a {\em homeomorphism}
in~\cite{Bx94}) if $f$ is a continuous bijection
such that $f^{-1}$ is continuous.

We use the following notation. For a
digital image $(X,\kappa)$,
\[ C(X,\kappa) = \{f: X \to X \, | \,
   f \mbox{ is continuous}\}.
\]

Given $f \in C(X,\kappa)$, a point
$x \in X$ is a {\em fixed point of $f$} if
$f(x)=x$. We denote by $\Fix(f)$ the set
$\{x \in X \, | \, x 
   \mbox{ is a fixed point of } f \}$.
A point $x \in X$ is an {\em almost fixed point}~\cite{Rosenfeld,TsSm} or
an {\em approximate fixed point}~\cite{BEKLL} of $f$ if $x \adjeq_{\kappa} f(x)$. Other papers in which approximate
fixed points were studied
include~\cite{BxNormal,BxAlt,BxApprox1,BxApprox2,BxConsequences,KangHan}.
The paper~\cite{KangHan} has inappropriate
citations and unoriginal results; these will be discussed
in section~\ref{HanSec}. However, one of the implications
of Theorem~4.4 of that paper is an important and original
contribution.

\subsection{Freezing and cold sets}
Material in this section is largely quoted or paraphrased from~\cite{BxFpSets}.

In a Euclidean space, knowledge of the fixed point set of a continuous self-map
$f: X \to X$ often gives little information about $f|_{X \setminus \Fix(f)}$. By contrast,
knowledge of $\Fix(f)$ for $f \in C(X,\kappa)$ can tell us much about $f|_{X \setminus \Fix(f)}$.
This motivates the study of freezing and cold sets.

\begin{definition}
\label{freezeDef}
{\rm \cite{BxFpSets}}
Let $(X,\kappa)$ be a digital image. We say $A \subset X$ is a 
{\em freezing set for $X$} if given $g \in C(X,\kappa)$,
$A \subset \Fix(g)$ implies $g=\id_X$. If no proper subset of a freezing set $A$ 
is a freezing set for $(X,\kappa)$, then $A$ is a {\em minimal freezing set}.
\end{definition}

\begin{defn}
\label{bdDef}
{\rm \cite{BxConvex}}
Let $X \subset \Z^n$.
\begin{itemize}
    \item The
{\em boundary of $X$ with respect to the $c_i$ adjacency},
$i \in \{1,2\}$, is
\[Bd_i(X) = \{x \in X \, | \mbox{ there exists } y \in \Z^n \setminus X \mbox{ such that } y \adj_{c_i} x\}.
\]
$Bd_1(X)$ is what is called the {\em boundary of $X$}
in~\cite{RosenfeldMAA}. This paper uses both $Bd_1(X)$ and $Bd_2(X)$.
\item The {\em interior of} $X$ with respect to the $c_i$ adjacency
is $Int_i(X) = X \setminus Bd_i(X)$.
\end{itemize}
\end{defn}

\begin{thm}
{\rm \cite{BxFpSets}}
\label{bdFreezes}
Let $X \subset \Z^n$ be finite. Then for $1 \le u \le n$, $Bd_1(X)$ is 
a freezing set for $(X,c_u)$.
\end{thm}

\begin{thm}
{\rm \cite{BxFpSets}}
\label{corners-min}
Let $X = \Pi_{i=1}^n [0,m_i]_{\Z}$.
Let $A = \Pi_{i=1}^n \{0,m_i\}$.
\begin{itemize}
\item Let $Y = \Pi_{i=1}^n [a_i,b_i]_{\Z}$ be
      such that $X \subset Y$. Let $f: X \to Y$ be
      $c_1$-continuous. If $A \subset \Fix(f)$, then $X \subset \Fix(f)$.
\item $A$ is a freezing set for $(X,c_1)$; minimal for $n \in \{1,2\}$.
\end{itemize}
\end{thm}

\begin{thm}
{\rm \cite{BxFpSets}}
\label{noProperSub}
Let $X = \prod_{i=1}^n[0,m_i]_{\Z}  \subset \Z^n$, where $m_i>1$ for all~$i$.
Then $Bd_1(X)$ is a minimal freezing set for $(X,c_n)$.
\end{thm}

In the following, we use the path-length metric $d$ for
connected digital images $(X,\kappa)$,
defined~\cite{Han05} as
\[ d_{\kappa}(x,y)= \min\{\ell \, | \, \ell
      \mbox{ is the length of a $\kappa$-path in $X$ from $x$ to $y$} \}.
\]
If $X$ is finite and $\kappa$-connected, the {\em diameter} of $(X,\kappa)$ is
\[ diam(X,\kappa) = \max\{d_{\kappa}(x,y) \, | \, x,y \in X\}.
\]

\begin{definition}
\label{s-cold-def}
{\rm \cite{BxFpSets}}
Given $s \in \N^*$, we say $A \subset X$ is an
{\em $s$-cold set} for the connected digital image $(X,\kappa)$
if given $f \in C(X,\kappa)$ such that
$f|_A = \id_A$, then for all $x \in X$, $d_{\kappa}(x,f(x)) \le s$.
If no proper subset of $A$ is an $s$-cold set
for $(X,\kappa)$, then $A$ is {\em minimal}.
A {\em cold set} is a 1-cold set.
\end{definition}

\begin{thm}
\label{s-cold-invariant}
{\rm \cite{BxFpSets}}
Let $(X,\kappa)$ be a connected digital image, 
let $A$ be an $s$-cold set for $(X,\kappa)$,
and let $F: (X,\kappa) \to (Y,\lambda)$ be an
isomorphism. Then $F(A)$ is
an $s$-cold set for $(Y,\lambda)$.
\end{thm}

\begin{remark}
{\rm \cite{BxFpSets}}
\label{coldRemark}
The following are easily observed.
\begin{enumerate}
    \item A 0-cold set is a freezing set.
    \item If $A \subset A' \subset X$ and $A$ is an $s$-cold set for
          $(X,\kappa)$, then $A'$ is an $s$-cold set for $(X,\kappa)$.
    \item $A$ is a cold set (i.e., a 1-cold set)
          for $(X,\kappa)$ if and only if given $f \in C(X,\kappa)$ 
          such that $f|_A = \id_A$, every $x \in X$ is an approximate fixed point of $f$.
    \item In a finite connected digital image $(X,\kappa)$, every nonempty
          subset of $X$ is a $diam(X)$-cold set.
    \item If $s_0 < s_1$ and $A$ is an $s_0$-cold set
          for $(X,\kappa)$, then $A$ is an $s_1$-cold set for $(X,\kappa)$.
\end{enumerate}
\end{remark}

\subsection{Digital disks and bounding curves}
Material in this section is largely quoted or paraphrased from~\cite{BxConvex}.

Let $\kappa \in \{c_1,c_2\}$, $n>1$. We say a $\kappa$-connected 
set $S=\{x_i\}_{i=1}^n \subset \Z^2$ is a
{\em (digital) line segment} if the members of $S$ are collinear.

\begin{remk}
\label{segSlope}
{\rm \cite{BxConvex}}
A digital line segment must be vertical, horizontal, or have slope of $\pm 1$.
\end{remk}

We say a segment with slope of $\pm 1$ is
{\em slanted}. An {\em axis-parallel} segment is horizontal or vertical.

A {\em (digital) $\kappa$-closed curve} is a
path $S=\{s_i\}_{i=0}^m$ such that $s_0=s_m$,
and $0 < |i - j| < m$ 
implies $s_i \neq s_j$. If, also, $0 \le i < n$ implies
\[ N(S,x_i,\kappa)=\{x_{(i-1)\mod n},x_{(i+1)\mod n}\}
\]
$S$ is a {\em (digital) $\kappa$-simple closed curve}.
For a simple closed curve $S \subset \Z^2$ we generally assume
\begin{itemize}
    \item $m \ge 8$ if $\kappa = c_1$, and
    \item $m \ge 4$ if $\kappa = c_2$.
\end{itemize}
These requirements are necessary for the Jordan Curve
Theorem of digital topology, below, as a
$c_1$-simple closed curve in $\Z^2$ must have at least 8 points to
have a nonempty finite complementary $c_2$-component,
and a $c_2$-simple closed curve in $\Z^2$ must have at least 4 points to
have a nonempty finite complementary $c_1$-component.
Examples in~\cite{RosenfeldMAA} show why it is
desirable to consider $S$ and $\Z^2 \setminus S$
with different adjacencies.

\begin{thm}
{\rm \cite{RosenfeldMAA}}
{\em (Jordan Curve Theorem for digital topology)}
Let $\{\kappa, \kappa'\} = \{c_1, c_2\}$.
Let $S \subset \Z^2$ be a simple closed 
$\kappa$-curve such that $S$ has at least 8 points if
$\kappa = c_1$ and such that $S$ has at least 
4 points if $\kappa = c_2$. Then
$\Z^2 \setminus S$ has exactly 2 $\kappa'$-connected
components.
\end{thm}

One of the $\kappa'$-components of 
$\Z^2 \setminus S$ is finite and the other is infinite. This 
suggests the following.
\begin{defn}
\label{diskDef}
{\rm \cite{BxConvex}}
Let $S \subset \Z^2$ be a $c_2$-closed curve such that
$\Z^2 \setminus S$ has two $c_1$-components, one finite and the
other infinite. The union $D$ of $S$ and the finite $c_1$-component 
of $\Z^2 \setminus S$ is a {\em (digital) disk}. $S$ is
a {\em bounding curve} of $D$. The finite $c_1$-component 
of $\Z^2 \setminus S$ is the {\em interior of} $S$, denoted $Int(S)$,
and the infinite $c_1$-component of $\Z^2 \setminus S$ is the {\em exterior of} 
$S$, denoted $Ext(S)$.
\end{defn}

Notes:
\begin{itemize}
    \item If $D$ is a digital disk determined as above by a bounding 
    $c_2$-closed curve $S$, then $(S,c_1)$ can be 
    disconnected. See Figure~\ref{fig:diamond}.
    \item There may be more than one closed curve $S$
          bounding a given disk~$D$. See Figure~\ref{fig:2sccBdry}.
          When $S$ is understood as a bounding curve of a disk $D$,
          we use the notations $Int(S)$ and $Int(D)$ interchangeably.
    \item Since we are interested in finding {\em minimal}
          freezing or cold sets and since it turns out we often compute these
          from bounding curves, we may prefer those of
          minimal size. A bounding curve~$S$
          for a disk $D$ is {\em minimal} if there is no
          bounding curve $S'$ for $D$ such that
          $\#S' < \#S$.
    \item In particular, a bounding 
          curve need not be contained in $Bd_1(D)$.
          E.g., in the disk~$D$
          shown in Figure~\ref{fig:2sccBdry}(i), $(2,2)$ is a point
          of the bounding curve; however, all of the points
          $c_1$-adjacent to $(2,2)$ are members of~$D$, so
          by Definition~\ref{bdDef}, $(2,2) \not \in Bd_1(D)$.
          However, a bounding curve for $D$ must be contained
          in $Bd_2(D)$.
    \item In Definition~\ref{diskDef}, we use $c_2$ adjacency for
          $S$ and we do not require $S$ to be simple. 
          Figure~\ref{fig:2sccBdry} shows why these seem
          appropriate.
          \begin{itemize}
              \item The $c_2$ adjacency allows slanted
              segments in bounding curves and makes possible
          a bounding curve in subfigure~(ii) with fewer points
          than the bounding curve in subfigure~(i) in which
          adjacent pairs of the bounding curve are restricted
          to $c_1$ adjacency.
          \item Neither of the bounding curves shown in
                Figure~\ref{fig:2sccBdry} is a $c_2$-simple closed
                curve. E.g., non-consecutive points of each of
                the bounding curves,
                $(0,1)$ and $(1,0)$, are $c_2$-adjacent. The
                bounding curve shown in 
                Figure~\ref{fig:2sccBdry}(ii) is clearly also not a
                $c_1$-simple closed curve.
          \end{itemize}
    \item A closed curve that is not simple may be the boundary~$Bd_2$
          of a digital image that is not a disk. This is illustrated
          in Figure~\ref{fig:notDisk}.
\end{itemize}

\begin{figure}
    \centering
    \includegraphics[height=1.5in]{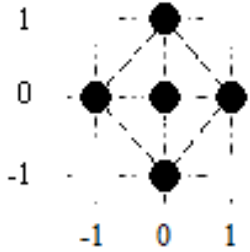}
    \caption{\cite{BxConvex} The $c_1$-disk
    $D = \{(x,y) \in \Z^2 \, | \, |x| + |y| < 2\}$.
    The bounding curve $S = \{(x,y) \in \Z^2 \, | \,
    |x| + |y| =1\} = D \setminus \{(0,0)\}$ is not $c_1$-connected.
    }
    \label{fig:diamond}
\end{figure}

\begin{figure}
    \centering
    \includegraphics[height=1.25in]{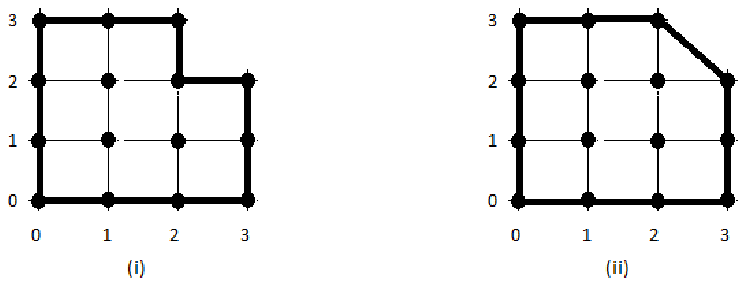}
    \caption{\cite{BxConvex} Two views of $D = [0,3]_{\Z}^2 \setminus \{(3,3)\}$, 
    which can be regarded as a $c_1$-disk with either of the
    closed curves shown in dark as a bounding curve. \newline
    (i) The dark line segments show a $c_1$-simple closed curve $S$
    that is a bounding curve for~$D$.
    Note the point $(2,2)$ in the bounding curve shown.
    By Definition~\ref{bdDef},
    $(2,2) \not \in Bd_1(D)$; however, $(2,2) \in Bd_2(D)$. \newline
    (ii) The dark line segments show a $c_2$-closed curve $S$
    that is a minimal bounding curve for~$D$.
       }
    \label{fig:2sccBdry}
\end{figure}

\begin{figure}
    \centering
    \includegraphics{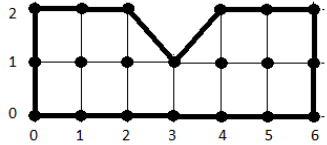}
    \caption{\cite{BxConvex} $D = [0,6]_{\Z} \times [0,2]_{\Z} \setminus \{(3,2)\}$
     shown with a bounding curve $S$ in dark segments. $D$ is
     not a disk with either the $c_1$ or the $c_2$ adjacency,
     since with either of these adjacencies,
     $\Z^2 \setminus S$ has two bounded components,
     $\{(1,1), (2,1)\}$ and $\{(4,1), (5,1)\}$.
     }
    \label{fig:notDisk}
\end{figure}

More generally, we have the following.
\begin{defn}
\label{boundingCurvesDef}
{\rm \cite{BxConvex}}
Let $X \subset \Z^2$ be a finite, $c_i$-connected set,
$i \in \{1,2\}$. Suppose there are 
pairwise disjoint $c_2$-closed curves
$S_j \subset X$, $1 \le j \le n$, such that
\begin{itemize}
    \item $X \subset S_1 \cup Int(S_1)$;
    \item for $j>1$, $D_j = S_j \cup Int(S_j)$ is a digital disk;
    \item no two of 
    \[ S_1 \cup Ext(S_1), D_2, \ldots, D_n
    \]
    are $c_1$-adjacent or $c_2$-adjacent; and
    \item we have 
    \[ \Z^2 \setminus X = Ext(S_1) \cup \bigcup_{j=2}^n Int(S_j).
    \]
\end{itemize}
Then $\{S_j\}_{j=1}^n$ is a {\em set of bounding curves of} $X$.
\end{defn}

Note: As above, a digital image $X \subset \Z^2$ may have more than one
set of bounding curves.

 \begin{figure}
        \centering
        \includegraphics[height=1in]{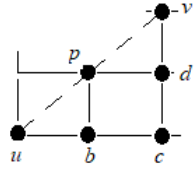}
        \caption{\cite{BxConvex}
        $p \in \overline{uv}$ in a bounding curve,
        with $\overline{uv}$ slanted.
        Note $u \not \adj_{c_1} p \not \adj_{c_1} v$,
        $p \adj_{c_2} c \not \adj_{c_1} p$,
        $\{p,c\} \subset N(\Z^2,c_1,b) \cap N(\Z^2,c_1,d)$. If
        $X$ is slant-thick at $p$ then $c \in X$.
        (Not meant to be understood as showing all of $X$.)}
        \label{fig:innerBdPt}
    \end{figure}

\subsection{Thickness}
A notion of ``thickness" in a digital image $X$, introduced in~\cite{BxConvex},
means, roughly speaking, $X$ is ``locally" like a disk.

Our definition of thickness depends on a notion of an
``interior angle" of a disk. We have the following.

\begin{definition}
\label{interiorAngleDef}
{\rm \cite{BxConvex}}
Let $s_1$ and $s_2$ be sides of a digital disk
$X \subset \Z^2$, i.e., maximal digital line segments
in a bounding curve $S$ of $X$, such that 
$s_1 \cap s_2 = \{p\} \subset X$.
The {\em interior angle of $X$ at $p$} is the
angle formed by $s_1$, $s_2$, and $Int(S)$.
\end{definition}

\begin{definition}
Let $X \subset \Z^2$ be a digital disk. Let $S$ be a bounding curve of $X$ and $p \in S$.
\begin{itemize}
    \item Suppose $p$ is in a maximal slanted segment $\sigma$ of $S$ such that
          $p$ is not an endpoint of $\sigma$. Then $X$ is {\em slant-thick at $p$} 
          if there exists $c \in X$ such that (see Figure~\ref{fig:innerBdPt})
          \begin{equation}
          \label{slantSegProp}
           c \adj_{c_2} p \not \adj_{c_1} c,
          \end{equation}
    \item Suppose $p$ is the vertex of a 90$^\circ$ ($ \pi / 2$ radians) interior angle
              $\theta$ of $S$. Then $X$ is {\em $90^\circ$-thick at $p$} if there
              exists $q \in Int(X)$ such that
              \begin{itemize}
                  \item if $\theta$ has axis-parallel sides then
                        $q \adj_{c_2} p \not \adj_{c_1} q$ (see Figure~\ref{fig:degrees90a}(1));
                  \item if $\theta$ has slanted sides then $q \adj_{c_1} p$ (see
                        Figure~\ref{fig:degrees90a}(2)).
              \end{itemize}
    \item Suppose $p$ is the vertex of a 135$^\circ$ ($3 \pi / 4$
          radians) interior angle $\theta$ of $S$. Then $X$ is
          {\em 135$^\circ$-thick at $p$} if there exist $b,b' \in X$
          such that $b$ and $b'$ are in the interior of $\theta$ and
          (see Figure~\ref{fig:degrees135c1})
          \[ b \adj_{c_2} p \not \adj_{c_1} b~~~ \mbox{ and }~~~ 
                b' \adj_{c_1} p.
          \]
\end{itemize}
\end{definition}

\begin{defn}
\label{thickness}
{\rm \cite{BxConvex}}
Let $X \subset \Z^2$ be a digital disk. We say $X$ is
{\em thick} if the following are satisfied. For some bounding
curve $S$ of $X$,
\begin{itemize}
    \item for every maximal slanted segment of~$S$,
          if  $p \in S$ is not an endpoint of  $S$, 
          then $X$ is slant-thick at $p$, 
          and
    \item for every $p$ that is the vertex of a 90$^\circ$ ($ \pi / 2$
          radians) interior angle $\theta$ of $S$, $X$ is {\em $90^\circ$-thick at $p$}, 
          and
\item for every $p$ that is the vertex of a 135$^\circ$ ($3 \pi / 4$
      radians) interior angle $\theta$ of $S$, $X$ is 135$^\circ$-thick at $p$.
\end{itemize}
\end{defn}

       \begin{figure}
        \centering
        \includegraphics[height=0.75in]{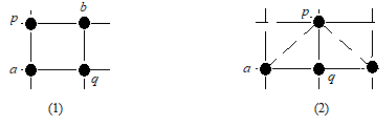}
        \caption{(1) $\angle apb$ is a
        $90^{\circ}$ ($\pi/2$ radians)
        angle of a bounding curve of $X$ at $p \in A_1$, with
        horizontal and vertical sides. If $X$ is $90^{\circ}$-thick
        at $p$ then $q \in Int(X)$. (Not meant to
        be understood as showing all of $X$.)\newline
        (2) $\angle apb$ is a
        $90^\circ$ ($\pi/2$ radians) angle
         between slanted segments of a bounding curve. If $X$ is
         $90^{\circ}$-thick at $p$ then $q \in Int(X)$. 
         (Not meant to be understood as showing all of $X$).}
        \label{fig:degrees90a}
    \end{figure}

      \begin{figure}
        \centering
        \includegraphics[height=1in]{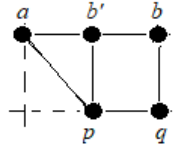}
        \caption{\cite{BxConvex} $\angle apq$ is an angle of
        135$^ \circ$ degrees ($3 \pi /4$ radians)
        of a bounding curve of $X$ at $p$, with
        $\overline{ap} \cup \overline{pq}$
            a subset of the bounding curve. If
            $X$ is $135^{\circ}$-thick at $p$ then $b,b' \in X$. (Not meant to
        be understood as showing all of $X$.)
        }
        \label{fig:degrees135c1}
        \end{figure}

\subsection{Convexity}
A set $X$ in a Euclidean space $\R^n$ is
{\em convex} if for every pair of distinct
points $x,y \in X$, the line segment
$\overline{xy}$ from $x$ to $y$ is contained in $X$.
The {\em convex hull of} $Y \subset \R^n$,
denoted $hull(Y)$, is the
smallest convex subset of $\R^n$ that contains~$Y$.
If $Y \subset \R^2$ is a finite set, then
$hull(Y)$ is a single point if $Y$ is a singleton;
a line segment if $Y$ has at least 2 members and all are
collinear; otherwise, $hull(Y)$ is a polygonal disk,
and the endpoints of the edges of $hull(Y)$ are its {\em vertices}.

A digital version of convexity can be stated
for subsets of the digital plane~$\Z^2$ as follows.
A finite set $Y \subset \Z^2$ is 
{\em (digitally) convex} {\rm \cite{BxConvex}} if either
\begin{itemize}
    \item $Y$ is a single point, or
    \item $Y$ is a digital line segment, or
    \item $Y$ is a digital disk with a bounding curve $S$
          such that the endpoints of the maximal line segments
          of~$S$ are the vertices of $hull(Y) \subset \R^2$.
\end{itemize}

\begin{remk}
{\rm \cite{BxConvex}}
Let $(X,\kappa)$ be a digital disk in $\Z^2$, 
$\kappa \in \{c_1,c_2\}$. Let $s_1$ and $s_2$ be sides of
$X$ such that $s_1 \cap s_2 = \{p\} \subset X$. Then
the interior angle of $X$ at $p$ is well defined.
\end{remk}

\begin{remk}
\label{convDiskIntAngles}
{\rm \cite{BxConvex}}
It follows from Remark~\ref{segSlope} that
every interior angle measures as a multiple
of 45$^\circ$ ($\pi/4$ radians). For a
convex disk, an interior angle must be
45$^\circ$ ($\pi/4$ radians),
90$^\circ$ ($\pi/2$ radians), or
135$^\circ$ ($3\pi/4$ radians).
\end{remk}

\section{Tools for determining fixed point sets}
The following assertions will be useful in determining fixed point and freezing
sets.

\begin{prop}
{\rm (Corollary 8.4 of~\cite{bs19a})}
\label{uniqueShortestProp}
Let $(X,\kappa)$ be a digital image and
$f \in C(X,\kappa)$. Suppose
$x,x' \in \Fix(f)$ are such that
there is a unique shortest
$\kappa$-path $P$ in~$X$ from $x$ 
to $x'$. Then $P \subseteq \Fix(f)$.
\end{prop}

Lemma~\ref{c1pulling}, below,
\begin{quote}
$\ldots$ can
be interpreted to say that
in a $c_u$-adjacency,
a continuous function that
moves a point~$p$ also moves
a point that is ``behind"
$p$. E.g., in $\Z^2$, if $q$ and $q'$ are
$c_1$- or $c_2$-adjacent with $q$
left, right, above, or below $q'$, and a
continuous function $f$ moves $q$ to the left,
right, higher, or lower, respectively, then
$f$ also moves $q'$ to the left,
right, higher, or lower, respectively~\cite{BxFpSets}.
\end{quote}

\begin{lem}
\label{c1pulling}
{\rm ~\cite{BxFpSets}}
Let $(X,c_u)\subset \Z^n$ be a digital image, 
$1 \le u \le n$. Let $q, q' \in X$ be such that
$q \adj_{c_u} q'$.
Let $f \in C(X,c_u)$.
\begin{enumerate}
    \item If $p_i(f(q)) > p_i(q) > p_i(q')$
          then $p_i(f(q')) > p_i(q')$.
    \item If $p_i(f(q)) < p_i(q) < p_i(q')$
          then $p_i(f(q')) < p_i(q')$.
\end{enumerate}
\end{lem}

\begin{remk}
\label{boundingCurveSetFreezes}
{\rm \cite{BxFpSets}}
If $X \subset \Z^2$ is finite, then
a set of bounding curves for $X$ is a freezing set
for $(X,c_i)$, $i \in \{1,2\}$.
\end{remk}

In particular, we have:

\begin{thm}
\label{bdCurveFreezes}
Let $D$ be a digital disk in $\Z^2$. Let
$S$ be a bounding curve for $D$. Then $S$ is
a freezing set for $(D,c_1)$ and for $(D,c_2)$.
\end{thm}

The next two results form a dual pair.

\begin{thm}
\label{convDiskThmActual}
{\rm \cite{BxConvex}}
Let $X$ be a thick convex disk with a
    bounding curve $S$.
    Let $A_1$ be the set of points $x \in S$ such that
$x$ is an endpoint of a maximal axis-parallel edge of $S$. Let $A_2$ 
be the union of slanted line segments in $S$.
Then $A = A_1 \cup A_2$ is a minimal 
freezing set for $(X,c_1)$.
\end{thm}

\begin{thm}
\label{convDiskThmC2Actual}
{\rm \cite{BxConvex}}
Let $X$ be a thick convex disk with a minimal bounding
curve $S$. Let $B_1$ be the set of
points $x \in S$ such that
$x$ is an endpoint of a maximal
slanted edge in $S$. Let $B_2$ 
be the union of maximal axis-parallel line segments in $S$.
Let $B = B_1 \cup B_2$. Then $B$ is a 
minimal freezing set for $(X,c_2)$.
\end{thm}

\section{General result}
\label{generalResultSec}
\begin{thm}
\label{1nbr}
Let $A$ be a cold set for the connected digital image $(X,\kappa)$.
Assume $\#X > 2$. 
Let $p \in X$  such that 
$\#N(X,p,\kappa) = 1$. 
Then $p \in A$.
\end{thm}

\begin{proof}
By hypothesis, there exists 
$p' \in X$ such that $\{p'\} = N(X,p,\kappa)$.
Since $X$ is connected and has more than 2 points,
there exists $q \in N(X,\kappa,p') \setminus \{p\}$.
Suppose $p \not \in A$. Then the function
$f: X \to X$ given by
\[ f(x) = \left \{ \begin{array}{ll}
    q &  \mbox{if } x = p; \\
    x &   \mbox{if } x \neq p,
\end{array}  \right .
\]
is a member of $C(X,\kappa)$; this follows from the observation that
\[ x \adj p ~~ \Rightarrow ~~ x = p' ~~ \Rightarrow ~~ f(x) = f(p') = p' \adj q = f(p).
\]
Clearly $f|_A = \id_A$ and 
$d_{\kappa}(p,f(p)) = d_{\kappa}(p,q) = 2$. The latter contradicts
the assumption that $A$ is a cold set. The assertion follows.
\end{proof}

\section{Results for vertices of boundary angles}
\label{essentialPtSec}
In this section, we state results concerning whether the
vertex of an interior angle formed 
by sides of a bounding curve must belong to a cold set.

\subsection{$45^{\circ}$ ($\pi / 4$ radians)}
\begin{prop}
\label{essentialCold45}
Let $X \subset \Z^2$. Let $S$ be a member of a set of 
minimal bounding curves for $X$. Let $a \in S$ be the vertex of an interior angle, with measure $45^{\circ}$
($\pi /4$ radians), formed by edges $E_1$ and $E_2$ of $S$.
Let $A$ be a freezing set or a cold set for $(X, c_1)$.
Then $a \in A$.
\end{prop} 

\begin{proof}
By Theorem~\ref{s-cold-invariant}, there is no loss of
generality in assuming 
$a = (0, 0)$, the points $(x, y)$ of $E_1$ satisfy 
$y = x \ge 0$, and the points of $E_2$ satisfy $x \ge 0 = y$.
The function $f: X \to X$ given by
\[ f(x) = \left \{ \begin{array}{ll}
   (1, 1) & \mbox{if } x = a; \\
    x & \mbox{if } x \neq a,
 \end{array} \right .
\]
is easily seen (see Figure~\ref{fig:innerBdPt}) to belong 
to $C(X, c_1)$. Further, if
$a \not \in A$ then $f|_A = id_A$ and $d_{c_1}(a, f(a)) = 2$,
the latter contrary to assumption if
$A$ is either a freezing set or a cold set for $(X, c_1)$.
The assertion follows.
\end{proof} 

\begin{prop}
\label{essentialc2cold45}
Let $X \subset \Z^2$. Let $S$ be a member of a set of minimal bounding curves for $X$.
Let $a \in S$ be the vertex of an interior angle, with measure
$45^{\circ}$ ($\pi / 4$ radians), formed by edges $E_1$ and $E_2$ of $S$. Let $p \in E_1$, $p \adj_{c_2} a$.
Let $X$ be slant-thick at $p$. Let $A$ be
a freezing set or a cold set for $(X,c_2)$. Then $a \in A$.
\end{prop}

\begin{proof}
By Theorem~\ref{s-cold-invariant}, there is no loss of 
generality in assuming $a = (0,0)$, the points $(x,y)$ of $E_1$ satisfy 
$y=x \ge 0$, and the points of $E_2$ satisfy $x \ge 0 = y$.

Since $X$ is  slant-thick at $p$, $c=(2,0) \in X$ (see Figure~\ref{fig:innerBdPt}).
Consider the function $f: X \to X$ given by
\[ f(x) = \left \{ \begin{array}{ll}
     c &  \mbox{if } x=a; \\
     x & \mbox{if } x \neq a.
\end{array}
  \right .
\]  
It is easily seen that $f \in C(X,c_2)$. Also, we have that $f|_A = \id_A$, and  
$d_{c_2}(a, f(a))=2$, so assuming $a \not \in A$ is contrary to the assumption that 
$A$ is a freezing  or cold set. The assertion follows.
\end{proof}

\begin{figure}
    \centering
    \includegraphics[height=1.75in]{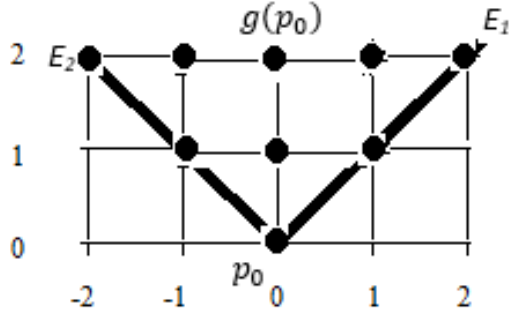}
    \caption{Illustration of the function $g$ of Proposition~\ref{essentialc2cold90}. All points
    of $X$ other than $p_0$ are fixed points of $f$. 
    Notice the point marked $p_0$
    is moved 2 units by $g$.\newline
    (Not to be understood as showing the entire image $X$.) The only $c_1$-neighbor of
         $p_0$ in $X$ is $(0,1)$, a fixed point of $g$ and a $c_1$-neighbor of $g(p_0)$,
         so $g \in C(X,c_1)$. \newline
         The points $(-1,1)$, $(0,1)$, and $(1.1)$ are the $c_2$-neighbors of $p_0$, are
         fixed points of $g$, and are $c_2$-neighbors of $g(p_0)$, so $g \in C(X,c_2)$.
         }
         \label{fig:degrees90openUp}
\end{figure}

\subsection{$90^{\circ}$ ($\pi / 2$ radians)}

\begin{prop}
\label{essentialc2cold90}
Let $X \subset \Z^2$. Let $S$ be a minimal bounding curve for $X$. 
Let $p_0$ be the vertex of an interior angle of $S$, formed by slanted edges $E_1$ and $E_2$ of $S$, 
of measure $90^{\circ}$ ($\pi / 2$ radians). Let $A$ be any of a freezing set for $(X,c_1)$, 
a cold set for $(X,c_1)$, a freezing set for $(X,c_2)$, or a cold set for $(X,c_2)$.
Let $X$ be $90^{\circ}$-thick at $p_0$. Then $p_0 \in A$.
\end{prop}

\begin{proof}
By Theorem~\ref{s-cold-invariant}, there is no loss of generality in assuming
$p_0=(0,0)$, points of $E_1$ satisfy $y = x \ge 0$, and
points of $E_2$ satisfy $y = -x \le 0$.

Since $X$ is $90^{\circ}$-thick at $p_0$, $q = (2,0) \in X$ (see Figure~\ref{fig:degrees90openUp}).
Suppose $p_0 \not \in A$. Consider the function $g: X \to X$ given by
\[ g(x) = \left \{ \begin{array}{ll}
     q &  \mbox{if } x=p_0; \\
     x & \mbox{if } x \neq p.
\end{array}
  \right .
\]  
It is easily seen that $g \in C(X,c_1)$ and
$g \in C(X,c_2)$. If $p \not \in A$ then $f|_A =\id_A$ and 
$d_{c_i}(p_0, g(p_0)) = 2$ for $i \in \{1,2\}$, contrary to 
the assumption that $A$ is a freezing or cold set. Therefore
we must have $p_0 \in A$.
\end{proof}

\begin{prop}
\label{90c1}
Let $X \subset \Z^2$. Let $S$ be a bounding curve for $X$.
Let $p$ be the vertex of an interior angle of $S$ formed by axis-parallel edges $E_1$ and $E_2$ of $S$,
of measure $90^{\circ}$ ($\pi / 2$ radians). Let $X$ be $90^{\circ}$-thick at $p$. Let $A$ be
a cold set for $(X,c_1)$. Then $p \in A$.
\end{prop}

\begin{proof}
By Theorem~\ref{s-cold-invariant}, we may assume $p=(0,0)$, points of $E_1$ satisfy
$x \ge 0$, $y=0$, and points of $E_2$ satisfy $x=0$, $y \le 0$.
Since $X$ is $90^{\circ}$-thick at $p$, $q=(1,-1) \in X$ (see Figure~\ref{fig:degrees90a}(1)).
Let $A$ be a cold set for $(X,c_1)$. Suppose $p \not \in A$.
Let $f: X \to X$ be the function given by
          \[ f(x) = \left \{ \begin{array}{ll}
             q  & \mbox{if }  x=p; \\
              x & \mbox{if }  x \neq p.
          \end{array}
                    \right .
          \]
          It is easily seen that $f \in C(X,c_1)$, $f|_A = \id_A$. However,
          $d_{c_1}(p,f(p)) = 2$, contrary to the assumption that $A$ is cold for $(X,c_1)$.
Therefore, we must have $p \in A$.
\end{proof}

We do not obtain a similar conclusion if $c_2$ is substituted for $c_1$ in the hypotheses of
Proposition~\ref{90c1}, as shown in the following example.

\begin{exl}
\label{90c2exl}
Let $X = [0,2]_{\Z}^2$. Then $p_0 = (0,0)$ is the vertex of an interior
angle of $90^{\circ}$ ($\pi /2$ radians) with axis-parallel sides, and $X$ is
$90^{\circ}$-thick at $p_0$,
but $p_0$ is not a member of every cold set for $(X,c_2)$.
\end{exl}

\begin{proof}
Let $A = X \setminus \{p_0\}$. Let $g \in C(X,c_2)$ such that $g|_A = \id_A$. 
Then continuity implies
\[ g(p_0)  \in N^*(X,(1,0),c_2) \cap N^*(X,(0,1),c_2) = \{p_0, (1,1) \}
   \subset N^*(X,p_0,c_2).
\]
Therefore, $A$ is a cold set for $(X,c_2)$.
\end{proof}

\subsection{$135^{\circ}$ ($3\pi / 4$ radians)}

\begin{prop}
\label{essentialc1cold135}
Let $X \subset \Z^2$ have a  $135^{\circ}$ ($3\pi / 4$ radians) interior angle at $p_0$.
Suppose $X$ is $135^{\circ}$-thick at $p_0$.
Then for every cold set $A$ for $(X,c_1)$, $p_0 \in A$.
\end{prop}

\begin{proof}
By Theorem~\ref{s-cold-invariant}, we may assume $p_0 = (0,0)$, points $(x,y) \in E_1$ satisfy
$x \ge 0 = y$, and points $(x,y) \in E_2$ satisfy $y = -x \ge 0$.

Suppose $p_0 \not \in A$. Since $X$ is $135^{\circ}$-thick at $p_0$, $b=(1,1) \in X$ (see Figure~\ref{fig:degrees135c1}).
The function $f: X \to X$ given by
\[ f(x) = \left \{ \begin{array}{cc}
    b & \mbox{if } x = p_0;  \\
    x & \mbox{if } x \neq p_0,
\end{array}    \right .
\]
is a member of $C(X,c_1)$, since $N(X, p_0, c_1) = \{(0,1), (1,0)\}$ and
\[ (0,1) = f(0,1) \adj_{c_1} f(p_0) \adj_{c_1} (1,0) = f(1,0).
\]
Also, $f|_A = \id_A$. However, $d_{c_1}(p_0, f(p_0)) = 2$, contrary to the
assumption that $A$ is cold. The contradiction yields the assertion.
\end{proof}

\begin{figure}
    \centering
    \includegraphics[height=1.5in]{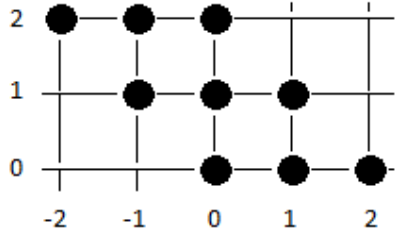}
    \caption{Interior angle of $135^{\circ}$ ($3 \pi / 4$ radians) at vertex $(0,0)$ in the
             image of Example~\ref{135c2Exl}.
    }
    \label{fig:degrees135a}
\end{figure}

If we replace $c_1$ with $c_2$ in Theorem~\ref{essentialc1cold135}, 
we do not obtain a similar conclusion, as shown in the following example.

\begin{exl}
\label{135c2Exl}
Let $X = ([0,2]_{\Z} \times \{0\}) \cup ([-1,1]_{\Z} \times \{1\}) \cup ([-2,0]_{\Z} \times \{2\})$
(see Figure~\ref{fig:degrees135a}). Then $(0,0)$ is the vertex
of an interior angle in $X$ measuring $135^{\circ}$ ($3 \pi /4$ radians),
$X$ is $135^{\circ}$-thick at $(0,0)$, and
$A = X \setminus \{(0,0)\}$ is a cold set for $(X,c_2)$.
\end{exl}

\begin{proof}
Let $f \in C(X,c_2)$ such that $f|_A = \id_A$. By continuity, we must have
\[ f(0,0) \in N^*(X, f(-1,1), c_2)  \cap N^*(X, f(1,0), c_2) =
\]
\[
   N^*(X, (-1,1), c_2)  \cap N^*(X, (1,0), c_2) = \{(0,0), (0,1) \} \subset N^*(X, (0,0), c_2).
\]
The assertion follows.
\end{proof}

However, we have the following.

\begin{prop}
\label{135c2Freeze}
Let $X$ be a digital disk in $\Z^2$ that is $135^{\circ}$-thick at 
$p$, where $p$ is the vertex of an interior angle of $X$
formed by edges $E_1$ and $E_2$ of a minimal 
bounding curve $S$ for $X$. Let $A$ be
a freezing set for $(X,c_2)$. Then $p \in A$.
\end{prop}

\begin{proof}
By Theorem~\ref{s-cold-invariant}, we may assume $p = (0,0)$, points $(x,y) \in E_1$ satisfy
$x \ge 0 = y$, and points $(x,y) \in E_2$ satisfy $y = -x \ge 0$.

Suppose there is a freezing set $A$ for $(X,c_2)$ such that $p \not \in A$. Since $X$
is $135^{\circ}$-thick at $p$, $b' = (0,1) \in N(X,p, c_2)$ (see Figure~\ref{fig:degrees135c1}).
Then the function $f: X \to X$ given by
\[ f(x) = \left \{ \begin{array}{ll}
   b' & \mbox{if } x=p; \\
   x  &  \mbox{if } x \neq p,
\end{array} \right .
\]
is easily seen to belong to $C(X,c_2)$, with $f|_A = \id_A$ and $d_{c_2}(p,f(p)) = 1$, 
contrary to the assumption that $A$ is freezing. The assertion follows.
\end{proof}

\subsection{$225^{\circ}$ ($5 \pi / 4$ radians)}
\begin{figure}
    \centering
    \includegraphics[height=2in]{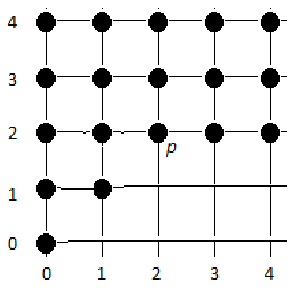}
    \caption{A digital image $X$ used for
    Example~\ref{225degreeExl}. \newline
   If $f \in C(X,c_1)$, $p=(2,2)$,
    $f(1,2) = (1,2)$, and $f(3,2) = (3,2)$, then we must have
    $f(p)=p$. Thus $X \setminus \{p\}$ is a freezing set, hence cold set, for $(X,c_1)$.
    }
    \label{fig:degrees225Vertex}
\end{figure}

The following example shows that a vertex of a $225^{\circ}$ 
($5 \pi / 4$ radians) angle
need not be a member of a given cold set. 

\begin{exl}
\label{225degreeExl}
Let $X = \{(0,0), (0,1), (1,1)\} \cup [0,4]_{\Z} \times [2,4]_{\Z}$ (see
Figure~\ref{fig:degrees225Vertex}). Let $p = (2,2)$. 
Note $p$ is a member of a bounding curve of $X$ and is a vertex
at which the interior angle is $225^{\circ}$ ($5 \pi /4$ radians).
If $A = X \setminus \{p\}$, then $A$ is a freezing set, hence 
a cold set, for both $(X,c_1)$ and $(X,c_2)$.
\end{exl}

\begin{proof}
Let $f \in C(X,c_1)$ be such that $f|_A = \id_A$. Then
\[ f(p) \in N^*(X, (1,2), c_1) \cap N^*(X, (3,2), c_1) = \{p\}.
\]
Thus $f = \id_A$, so $A$ is a freezing set, hence a cold set for $(X,c_1)$.

Let $f \in C(X,c_2)$ be such that $f|_A = \id_A$. Then
\[ f(p) \in N^*(X, (1,1), c_2) \cap N^*(X, (3,3), c_2) = \{p\}.
\]
Thus $f = \id_A$, so $A$ is  a freezing set, hence a cold set for $(X,c_2)$.
\end{proof}

\subsection{$270^{\circ}$ ($3 \pi /2$ radians)}
The following examples show that the vertex of an interior 
angle that measures
$270^{\circ}$ ($3 \pi /2$ radians) need not belong to a
given freezing, hence cold, set
for its digital image when either the $c_1$ or the $c_2$
adjacency is used.

\begin{exl}
\label{unionRectanglesExl}
{\rm
Let $X = ([0,2]_{\Z} \times [0,2]_{\Z}) \cup
          ([2,4]_{\Z} \times[0,3]_{\Z})$
(see Figure~\ref{fig:unionRectangles1}).
A minimal freezing set, and therefore a cold set, for $(X,c_1)$ is
\[ A = \{(0,0), (4,0), (4,3), (2,3), (0,2) \} \mbox{ \cite{BxSubsets}}.
\]
A freezing set, and therefore a cold set, for $(X,c_2)$ is,
by Theorem~\ref{convDiskThmC2Actual},
\[ B = \{(0,i)\}_{i=0}^2 \cup \{(j,0)\}_{j=0}^4 \cup \{(4,k)\}_{k=0}^3 \cup
        \{1,2), (2,3), (3,3)\}.
\]

The point $p=(2,2)$, at which $X$ has an internal angle of
       $270^{\circ}$ ($3 \pi /2$ radians), is not a member of $A$, nor of $B$.
Note $p$ is also not a member of the minimal bounding curve of $X$, which
bypasses $p$ by using the diagonal path $\{(1,2), (2,3)\}$; in general,
a vertex of a bounding curve of an image $Y \in \Z^2$ at which the interior angle is
$270^{\circ}$ ($3 \pi /2$ radians),
is not a member of the minimal bounding curve of $Y$.
}
\end{exl}

\begin{figure}
    \centering
    \includegraphics[height=2in]{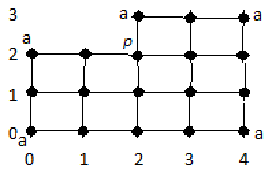}
    \caption{\cite{BxSubsets} The digital image $X$ of
    Example~\ref{unionRectanglesExl}. Points
    of a cold set $A$ for $(X,c_1)$ are marked ``a". Notes: \newline
    1) The point $p=(2,2)$, at which $X$ has an internal angle of
       $270^{\circ}$ ($3 \pi /2$ radians), is not a member of $A$. \newline
    2) The point $p$ does not belong to the minimal bounding curve, since the
       $c_1$-path $\{ (1,2), p, (2,3)\}$ of the $c_1$-bounding curve can be
       replaced by the $c_2$-path $\{ (1,2), (2,3)\}$ to obtain the minimal bounding curve.
       }
    \label{fig:unionRectangles1}
\end{figure}

\begin{exl}
\label{degrees270SlantExl}
Let $X = [0,4]_{\Z}^2 \setminus \{(1,0), (2,0), (2,1), (3,0)\}$
(see Figure~\ref{fig:degrees270slant}).
The point $p=(2,2)$ is the vertex of an interior angle 
of $270^{\circ}$ ($3 \pi /2$ radians) with slanted sides,
and does not belong to every freezing, hence cold, set for
$(X,c_1)$ or for $(X,c_2)$.
\end{exl}

\begin{proof}
Let $A = X \setminus \{p\}$. We will show $A$ is a freezing set, hence a cold set,
for both $(X,c_1)$ and $(X,c_2)$.

Let $f \in C(X,c_1)$ be such that $f|_A = \id_A$. Then
\[ f(p) \in N^*(X, f(1,2), c_1) \cap N^*(X, f(2,3), c_1) \cap N^*(X, f(3,2), c_1) = 
\]
\[    N^*(X, (1,2), c_1) \cap N^*(X, (2,3), c_1) \cap N^*(X, (3,2), c_1) =  \{p\}.
\]
Thus $f = \id_X$, so $A$ is a freezing set, hence a cold set, for $(X,c_1)$.

Let $f \in C(X,c_2)$ be such that $f|_A = \id_A$. Then
\[ f(p) \in N^*(X, f(1,1), c_2) \cap N^*(X, f(2,3), c_2) \cap N^*(X, f(3,1), c_2) = \]
\[ N^*(X, (1,1), c_2) \cap N^*(X, (2,3), c_2) \cap N^*(X, (3,1), c_2) = \{p\}.
\]
Thus $f = \id_X$, so $A$ is a freezing set, hence a cold set, for $(X,c_2)$.
\end{proof}

\begin{figure}
    \centering
    \includegraphics[height=2in]{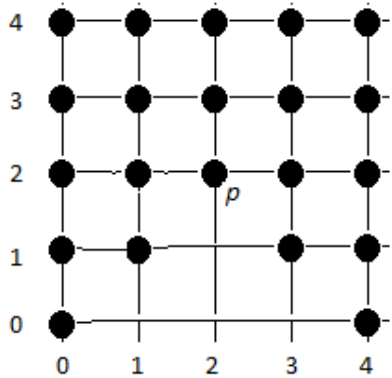}
    \caption{The digital image $X$ of
    Example~\ref{degrees270SlantExl}. The point $p=(2,2)$ is the vertex
    of an interior angle of $270^{\circ}$ ($3 \pi / 2$ radians).
       }
    \label{fig:degrees270slant}
\end{figure}

\subsection{$315^{\circ}$ ($7 \pi /4$ radians)}
\begin{exl}
If $X \subset \Z^2$ is a thick convex digital disk and $p \in X$ is the vertex of an angle in $X$
of measure $315^{\circ}$ ($7 \pi /4$ radians), then $p \in Int(S)$ for any bounding curve $S$ of $X$.
There are cold sets for both the $c_1$ and the $c_2$ 
adjacencies that do not contain $p$.
\end{exl}

\begin{proof}
Figure~\ref{fig:degrees315} shows how in 
a thick convex digital disk $X$ with $p \in X$ as the vertex of an angle of
$315^{\circ}$ ($7 \pi /4$ radians), $p$ must belong to $Int(S)$ for any bounding curve $S$ of $X$. By
Theorems~\ref{convDiskThmActual} and~\ref{convDiskThmC2Actual}, $X$ has freezing sets,
hence cold sets, for both the $c_1$ and the $c_2$ adjacencies, that do not contain $p$.
\end{proof}

\begin{figure}
    \centering
    \includegraphics[height=1.5in]{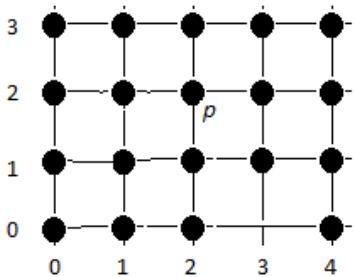}
    \caption{
      The interior angle with sides from $p$ to $(2,0)$ and from $p$ to $(4,0)$
             measures $315^{\circ}$ ($7 \pi /4$ radians). Note $p$ is an interior point
             of the digital image shown.
             }
    \label{fig:degrees315}
\end{figure}

\section{Results for $c_1$ adjacency in $\Z^2$}
\label{c1Sec} 
In this section, we obtain results for cold sets of digital images $X \subset \Z^2$
with respect to the $c_1$ adjacency.

\begin{thm}
\label{c1ConvexVertex}
Let $X$ be a thick convex digital disk in $\Z^2$. Let $S$ be a minimal bounding curve for $X$.
Let $p_0$ be a vertex of $hull(X)$. Let $A$ be a cold set
for $(X,c_1)$. Then $p_0 \in A$.
\end{thm}

\begin{proof}
Since $X$ is convex, the interior angle of $S$ at $p_0$ must be $45^{\circ}$
($\pi / 4$ radians), $90^{\circ}$ ($\pi / 2$ radians), or $135^{\circ}$ ($3 \pi / 4$ radians).
The assertion follows from Propositions~\ref{essentialCold45},~\ref{essentialc2cold90},~\ref{90c1}, 
and~\ref{essentialc1cold135}.
\end{proof}

\begin{figure}
    \centering
    \includegraphics[height=2in]{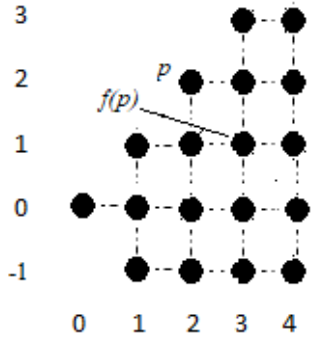}
    \caption{Example of the $c_1$-continuous function $f$ of Proposition~\ref{slantIn-c1}.
             A point $p$ of a slanted edge is marked, as is the point $f(p)$. All other points of the
             image $X$ are fixed points of $f$. Note that $d_{c_1}(p,f(p)) = 2$.
        }
    \label{fig:rectangleSlantEdges}
\end{figure}

\begin{prop}
\label{slantIn-c1}
Let $X$ be a thick digital disk in $\Z^2$ with bounding curve $S$. Let $\sigma$ be
a slanted edge of $S$. Let $p \in \sigma$ such that $p$ is not an endpoint of $\sigma$. Let
$A$ be a cold set for $(X,c_1)$. Then $p \in A$.
\end{prop}

\begin{proof}
By choice of $p$, there exists $q \in Int(X)$ such that $p \adj_{c_2} q$ and $p \not \adj_{c_1} q$.
We must have $p \in A$, for otherwise the function $f: X \to X$ given by
\[ f(x) = \left \{ \begin{array}{ll}
    q & \mbox{if } x = p; \\
    x & \mbox{if } x \neq p,
\end{array}           \right .
\]
(see Figure~\ref{fig:rectangleSlantEdges}) belongs to $C(X,c_1)$, $f|_A = \id_A$, and 
$d_{c_1}(p, f(p)) = 2$, contrary to the assumption that $A$ is cold.
\end{proof}

\begin{thm}
\label{c1coldEquivFreeze}
Let $X$ be a thick convex disk in $\Z^2$. Let $A \subset X$. Then $A$ is a cold set for
$(X,c_1)$ if and only if $A$ is a freezing set for $(X,c_1)$.
\end{thm}

\begin{proof}
Let $A$ be a cold set for $(X,c_1)$. Let $S$ be a minimal bounding curve for $X$.
From Theorem~\ref{c1ConvexVertex}, we know that 
\begin{equation}
\label{c1endsIn}
\mbox{the endpoints of edges of } S \mbox{ belong to $A$.}
\end{equation}
It follows from Proposition~\ref{slantIn-c1} that
\begin{equation}
\label{c1slantedEdgesIn}
\mbox{every slanted edge of $S$ is a subset of $A$.}
\end{equation}
It follows from~(\ref{c1endsIn}), (\ref{c1slantedEdgesIn}), and Theorem~\ref{convDiskThmActual} 
that $A$ is a freezing set for $(X,c_1)$.

The converse follows from Remark~\ref{coldRemark}(1),(5).
\end{proof}

\section{Results for $c_2$ adjacency in $\Z^2$}
\label{c2Sec}
In this section, we obtain results for cold sets of digital images $X \subset \Z^2$
with respect to the $c_2$ adjacency.

\begin{prop}
\label{c2coldEquivFreeze}
Let $X$ be a 4-sided thick digital disk in $\Z^2$, all sides of which are slanted. Let
$A$ be the set of endpoints of the edges of $X$. Then $A$ is a minimal cold set for $(X,c_2)$.
\end{prop}

\begin{proof}
By Theorem~\ref{convDiskThmC2Actual}, $A$ is a freezing set, hence a cold set for $(X,c_2)$. 
It follows from Proposition~\ref{essentialc2cold90} that $A$ is minimal as a cold set.
\end{proof}

\begin{prop}
\label{c2-coldSetRectangle}
{\rm \cite{BxFpSets}}
Let $m,n \in \N$. Let
$X = [0,m]_{\Z} \times [0,n]_{\Z}$.
Let $A \subset Bd_1(X)$ be such that
no pair of $c_1$-adjacent members of $Bd_1(X)$ 
belong to $Bd_1(X) \setminus A$.
Then $A$ is a cold set for $(X,c_2)$. Further,
for all $f \in C(X,c_2)$, if $f|_A = \id_A$ then
$f|_{Int(X)} = \id|_{Int(X)}$.
\end{prop}

We extend Proposition~\ref{c2-coldSetRectangle} as follows.

\begin{thm}
\label{c2-coldSetAxisParallel}
Let $X$ be a thick disk in $\Z^2$ 
with bounding curve $S$ made up of axis-parallel segments.
Let $A \subset S$ be such that
\begin{equation}
\label{noPair}
    \mbox{no pair of $c_1$-adjacent members of $S$ 
belong to $S \setminus A$.}
\end{equation}
Then $A$ is a cold set for $(X,c_2)$. Further,
for all $f \in C(X,c_2)$ such that $f|_A = \id_A$, we have
$f|_{Int(X)} = \id|_{Int(X)}$.
\end{thm}

\begin{proof}
Let $f \in C(X,c_2)$ be such that $f|_A = \id_A$. Let $x \in X$. We
must show $x$ is an approximate fixed point of $f$, i.e., that
$x \adjeq_{c_2} f(x)$. We consider the following cases.
\begin{itemize}
    \item If $x \in A$ then $x=f(x)$.
    \item If $x \in S \setminus A$ then since $S$ has only 
          axis-parallel segments, by~(\ref{noPair}), there are 
          distinct $y_0, y_1 \in A$ such that 
          $y_0 \adj_{c_1} x \adj_{c_1} y_1$.
          Therefore, either $y_0,x$, and $y_1$ are collinear or these points
          form a right angle at $x$. In either case, the 
          continuity of $f$ implies 
          \[ y_0 = f(y_0) \adjeq_{c_2} f(x) \adjeq_{c_2} f(y_1)=y_1.
          \]
          Thus, $f(x) \in N(X,y_0,c_2) \cap N(X,y_1,c_2)$, which implies
          $x \adjeq_{c_2} f(x)$. 
    \item If $x \in Int(X)$ then we create a $c_2$-path $P=\{q_i\}_{i=0}^n$ 
          through $x$ such that $q_i \adj_{c_2} q_{i+1}$ and
          $p_1(q_{i+1}) = p_1(q_i)+1$ for $i \in \{0, \ldots, n-1\}$, 
          and $\{q_0,q_n\} \subset A$. This is
          done as follows. Let $L$ be a minimal horizontal line segment 
          with the properties of containing $x$ and having
          endpoints in $S$.
          \begin{itemize}
              \item If the endpoints of $L$ are both members of $A$, take
              $P=L$ with $q_0, q_n$ as these endpoints.
              \item (See Figure~\ref{fig:TeeWithColdSet}(i).) 
                    If an endpoint of the horizontal through $x$ is in a 
                    vertical segment of $S$ and not in $A$, by~(\ref{noPair}) 
                    we can replace the endpoint with one of the adjacent members of $S$.
              \item (See Figure~\ref{fig:TeeWithColdSet}(ii).) 
                    If an endpoint of the horizontal through $x$ meets $S \setminus A$ 
                    at the vertex of a right angle, by~(\ref{noPair}) 
                    we can replace the endpoint with the adjacent member of the incident vertical side of $S$.
          \end{itemize}
          In all cases, the resulting path $P$ has endpoints in $A$ and is 
          monotone increasing, from left to right, in the first coordinate. By Lemma~\ref{c1pulling},
          $p_1(f(x)) = p_1(x)$.
          
          Similarly, $p_2(f(x))=p_2(x)$. Thus, $f(x)=x$.
\end{itemize}
Thus we have shown that for all $x \in X$, $f(x) \adjeq_{c_2} x$;
and $f|_{Int(X)}=\id_{Int(X)}$.
\end{proof}

\begin{figure}
    \centering
    \includegraphics[height=2in]{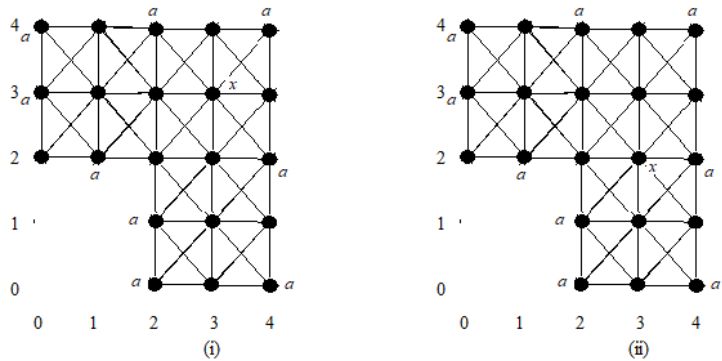}
    \caption{A digital disk $X$ with axis-parallel boundary segments 
    to illustrate Theorem~\ref{c2-coldSetAxisParallel}.
    Points of the set $A$ are labeled ``{\em a}". No $c_1$-adjacent pair of members of the bounding curve $S$ belong to $S \setminus A$. Given $x \in Int(X)$, we
    want to find a $c_2$-path between members of $A$ through $x$ that is monotone
    increasing in the first coordinate. This is easy if a horizontal segment
    in $X$ through $x$ has endpoints in $A$. Otherwise:\newline
    (i) If an endpoint of the horizontal through $x$ is in a vertical segment of
        $S$ and not in $A$, replace the endpoint with one of the adjacent members
        of $S$. E.g., with $x=(3,3)$, replace $(4,3)$ with $(4,2)$. This yields
        the $c_2$-path $\{(i,3)\}_{i=0}^3 \cup \{(4,2)\}$.        \newline
    (ii) If an endpoint of the horizontal through $x$ meets 
         $S \setminus A$ at the vertex of a right angle, replace the
         endpoint with the adjacent member of the incident vertical
         side of $S$. E.g., with $x=(3,2)$, replace $(2,2)$ 
         with $(2,1)$. This yields the $c_2$-path
         $\{(2,1), (3,2), (4,2)\}$.
        }
    \label{fig:TeeWithColdSet}
\end{figure}

\section{More on the choice of adjacency}
\label{contrastAdjs}
Example~\ref{c2Square-n-cold} below shows the importance of the adjacency used, since
by Theorems~\ref{corners-min} and~\ref{s-cold-invariant}, for the
same sets $X$ and $A$, with the $c_1$ adjacency, $A$ is a freezing set.

\begin{exl}
\label{c2Square-n-cold}
Let $X = [-n,n]_{\Z}^2$ for $n \ge 1$. Let 
\[A = \{(-n,-n), (-n,n), (n, -n), (n,n)\}.
\]
Then, for $(X,c_2)$, $A$ is an $n$-cold set and not an $(n-1)$-cold set.
\end{exl}

\begin{proof}
Let $f \in C(X,c_2)$ such that $f|_A = \id_A$. By Proposition~\ref{uniqueShortestProp}, the diagonals
\begin{equation} 
\label{diagEq}
D = \{(x,y) \in X \, | \, y = \pm x \} \subset \Fix(f).
\end{equation}
Consider the digital triangle of points,
\[ T_1 = \{(x,y) \in X \, | \, 0 \le x \le n, ~ 0 \le y \le x \}.
   \]
(see Figure~\ref{fig:c2coldOnBd-a}).

\begin{figure}
    \centering
    \includegraphics[height=3in]{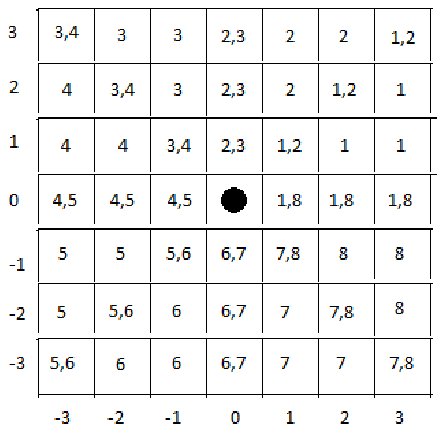}
    \caption{$X=[-n,n]_{\Z}^2$ (shown for $n=3$) for Example~\ref{c2Square-n-cold}.
     Pixels (other than the origin) labeled $i$ belong to triangle $T_i$. The origin belongs to $\bigcap_{i=1}^8 T_i$; other members of $D$ (the union of the two diagonals)
     each belong to two distinct $T_i$.\newline
           }
    \label{fig:c2coldOnBd-a}
\end{figure}

 Let $q=(u,v) \in T_1$. We will show that
 \begin{equation}
     \label{p_1Restricted}
     |u - p_1(f(q))| \le n.
 \end{equation}
Suppose otherwise. Then $u - p_1(f(q)) > n$. Let $q_1 = (n,v)$. By the $c_1$-continuity of $f$,
it follows from Lemma~\ref{c1pulling} that $p_1(f(q_1)) < 0$. Therefore,
$d_{c_2}(f(q_1), f(n,n)) = d_{c_2}(f(q_1), (n,n)) > n$ although
$d_{c_2}(q_1, (n,n)) = n-v \le n$. This is a contradiction, since $f \in C(X,c_2)$.
Thus~(\ref{p_1Restricted}) is established.

Next, we show
\begin{equation}
    \label{p_2Restricted}
    p_2(f(u,v)) = v.
\end{equation}
This follows from Lemma~\ref{c1pulling}, since $p_2(f(u,v)) < v$ would imply $p_2(f(u,u)) < u$, contrary to
$(u,u) \in \Fix(f)$; and $p_2(f(u,v)) > v$ would imply $p_2(f(u,-u)) > -u$, contrary to
$(u,-u) \in \Fix(f)$.

From~(\ref{p_1Restricted}) and~(\ref{p_2Restricted}), it follows that
$d_{c_2}(q,f(q)) \le n$. 

Similarly, $d_{c_2}(q,f(q)) \le n$ for $q$ a member of each of the following
subsets of $X$.
\[ T_2 = \{(x,y) \in X \, | \, 0 \le x \le y \le n \}, 
\]
\[ T_3 = \{(x,y) \in X \, | \, -n \le x \le 0, -x \le y \le n \},
\]
\[ T_4 = \{(x,y) \in X \, | \,  -n \le x \le 0 \le y \le -x \},
\]
\[ T_5 = \{(x,y) \in X \, | \,  -n \le x \le y \le 0 \},
\]
\[ T_6 = \{(x,y) \in X \, | \,  -n \le y \le x \le 0 \},
\]
\[ T_7 = \{(x,y) \in X \, | \, 0 \le x \le n, -n \le y \le -x \},
\]
\[ T_8 = \{(x,y) \in X \, | \, 0 \le x \le n, -x \le y \le 0 \}.
\]
Since $X = \bigcup_{i=1}^8 T_i$, we have $d_{c_2}(q,f(q)) \le n$ for all $q \in X$.
It follows that $A$ is an $n$-cold set for $(X,c_2)$.

To see $A$ is not an $(n-1)$-cold set for $(X,c_2)$, let
$g: X \to D \subset X$ be defined for $q=(u,v)$ by
\[ g(q) = \left \{ \begin{array}{ll}
    (u,|u|) & \mbox{if }  q \in T_2 \cup T_3; \\
    q & \mbox{otherwise} \\
\end{array} \right .
\]
(see Figure~\ref{fig:c2coldOnBd-b}).
Roughly, $g$ projects $T_2 \cup T_3$ vertically to $D$ and leaves all other points of $X$ fixed.
It is easy to see that $g \in C(X,c_2)$, $g|_A = \id_A$, and 
\[ d_{c_2}((0,n), g(0,n)) = d_{c_2}((0,n), (0,0)) = n.
\]
Thus $A$ is not an $(n-1)$-cold set for $(X,c_2)$.
\end{proof}

\begin{figure}
    \centering
    \includegraphics[height=3in]{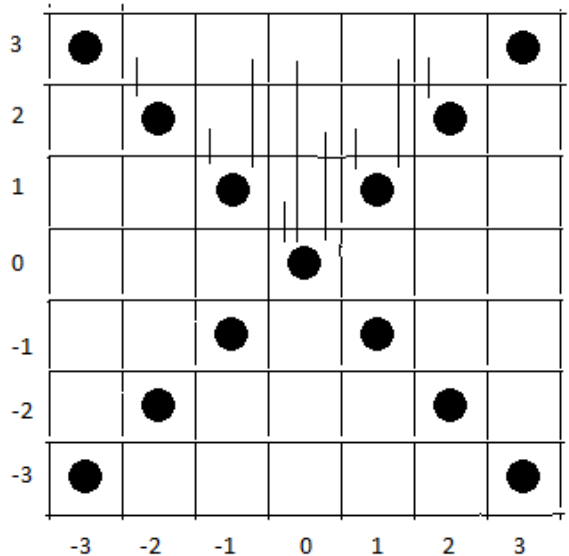}
    \caption{$X=[-n,n]_{\Z}^2$ (shown for $n=3$) for Example~\ref{c2Square-n-cold}.
     Illustration of the function $g$. Members of $D$ are shown dotted.
        Each $q \in X \setminus \Fix(g)$ is shown joined to $g(q) \in D$ by a line segment.
           }
    \label{fig:c2coldOnBd-b}
\end{figure}

\section{Remarks on~\cite{KangHan}}
\label{HanSec}
Much of this section is quoted or paraphrased from a 
comment posted to the website researchgate.com 
on the paper~\cite{KangHan}, by Kang and Han.
The paper has
a result that is original, correct and correctly proven, and
interesting; however, the paper is greatly flawed.
Some papers cited in~\cite{KangHan} should not
be rewarded with automated incremented citation counts,
so papers cited in this section but not otherwise cited 
in the current paper are listed at the end of this section
rather than among the references at the end of this paper.
Papers cited elsewhere in the current paper are referenced 
as elsewhere in the current paper.

The paper~\cite{KangHan} misguides readers through much
of the literature of the almost
fixed point property, also known as the approximate
fixed point property (AFPP), for digital images. The AFPP
generalizes the familiar fixed point property (FPP).

On page 7216, we find ``$\ldots$ any digital space $(X, k)$
on $\Z^n$ does not have the FPP for (digitally) $k$-continuous 
maps~\cite{Rosenfeld}
(for more details, see [H19, H20])." This statement is
obviously false if $X$ has a single point;
it is true when $X$ has more than one point (indeed, this case
is noted on page 7218), the proof appearing 
in~\cite{BEKLL}, a paper not cited in~\cite{KangHan}. 
While~\cite{Rosenfeld} gives examples of digital images
$(X, k)$ that lack the FPP, it says nothing like the general
statement above attributed to it by~\cite{KangHan}. Further,
Han's papers [H19, H20] contribute nothing to our knowledge of
this assertion. Variants on these errors appear
in the last paragraph of page 7218 of~\cite{KangHan},
where we also find Han's paper~[H17] falsely credited as
a source for the FPP assertion above.

Also on page 7216, we find ``$\ldots$ Banach contraction principle and a
Cauchy sequence for complete metric spaces, we have also 
studied this issue$\ldots$" followed by citation of several papers.
However, Cauchy sequences for digital metric spaces were shown
in Han's own paper~[H16] to be trivial, and
the Banach contraction principle was shown
in [BxSt] to be trivial for digital images when the most
natural metrics, including all $\ell_p$ metrics, are used. 
Indeed, most of the ``contributions" for digital metric spaces
in papers cited by~\cite{KangHan} were
shown in [BxSt] to be either trivial or incorrect.

Example 3.2(1)  of~\cite{KangHan} has no
originality, being a case of Theorem~3.3 of~\cite{Rosenfeld}. 

Example 3.2(2)  of~\cite{KangHan} has no
originality in either of its parts.
Part~(1-1) is implied by Theorem~3.5 of~\cite{BxApprox1}.
Both parts, (1-1) and~(1-2), are implied by
Theorem~4.8 of~\cite{BxApprox2}.

The adjacency given at Definition 4.1 of~\cite{KangHan},
attributed to Han's paper~\cite{Han05},
is the normal product adjacency. It was in the literature 
for decades before~\cite{Han05} appeared; see, e.g., [Ber].

Example~4.1 of~\cite{KangHan} should have been derived as an
immediate consequence of Theorem~3.3 of~\cite{Rosenfeld}
and Theorem~4.2 of~\cite{BEKLL}.

The only significant original contribution of~\cite{KangHan} 
is its Theorem~4.4. The assertion that the normal product of 
$(X,k_1)$ and $(Y,k_2)$ has
the AFPP if the factors have the AFPP, is correctly proven. 
However, the implication that the factors have the AFPP if
the product with the normal product adjacency has the AFPP, 
follows immediately and more simply from the fact that 
retractions preserve the AFPP~\cite{BEKLL}. Also, it should be noted
that Theorem 4.4 generalizes Theorem 4.5 of~\cite{BxApprox2}.

Theorem 4.9 of~\cite{KangHan} is unoriginal. See Theorem~4.8
of~\cite{BxApprox2}. \newline

References\newline

[Ber] C. Berge, Graphs and Hypergraphs, 2nd edition. North-Holland,
Amsterdam, 1976.\newline


[BxSt] L. Boxer and P.C. Staecker, Remarks on Fixed Point Assertions
in Digital Topology, Applied General Topology 20 (1) (2019), 135-153.\newline

[H16] S.-E. Han, Banach fixed point theorem from the viewpoint of
digital topology,
{\em Journal of Nonlinear Science and Applications}
9 (2016), 895 - 905\newline

[H17] S.-E. Han, Fixed point property for digital spaces, 
{\em Journal of Nonlinear Sciences and Applications}, 10 (2017), 2510-2523\newline

[H19] S.-E. Han, Remarks on the preservation of the almost
fixed point property involving several types
of digitizations, {\em Mathematics}, 7 (2019), 954.\newline

[H20] S.-E. Han, Digital $k$-contractibility of an $n$-times 
iterated connected sum of simple closed k-surfaces
and almost fixed point property, {\em Mathematics}, 8 (2020), 345
\newline

\section{Further remarks}

We have studied properties of cold sets for digital images $(X,\kappa)$ in the digital plane.
In sections~\ref{generalResultSec} and~\ref{essentialPtSec}, we have considered
essential members of cold sets for $(X,\kappa)$. In sections~\ref{c1Sec} 
and~\ref{c2Sec}, for the $c_1$ and $c_2$ adjacencies, respectively, we have
derived cold sets for thick digital disks $X$, with particular attention to
convex disks. In section~\ref{contrastAdjs}, we showed that the same sets 
$X \subset \Z^2$, $A \subset X$, can have very different cold-set properties
depending on whether the adjacency considered is $c_1$ or $c_2$.

The suggestions of an anonymous reviewer are acknowledged with
gratitude.

\end{document}